\documentclass{article}

\usepackage{amssymb,amsmath, amsthm}

\usepackage{url}

\usepackage{setspace} 
\usepackage{graphicx} 
\usepackage{color}  
\usepackage{pifont}
\usepackage[mathscr]{euscript}

\usepackage[T1]{fontenc}
\usepackage{epigraph} 
\usepackage[bottom]{footmisc}

\usepackage[lighttt]{lmodern}
\usepackage{caption}
\usepackage{booktabs} 
\usepackage{geometry}

\usepackage{titlesec}

\usepackage{tikz-cd}

\newtheorem{definition}{Definition}[section]
\newtheorem{proposition}{Proposition}[section]
\newtheorem{theorem}{Theorem}[section]
\newtheorem{corollary}{Corollary}[section]
\newtheorem{lemma}{Lemma}[section]
\newtheorem{example}{Example}[section]

\def\indep{\mathrel{\raise0.2ex\hbox{\ooalign{\hidewidth$\vert$\hidewidth\cr\raise-0.9ex\hbox{$\smile$}}}}}

\begin{document}

	\begin{center}
		\doublespacing	{\textbf{\huge{When are Two Subgroups Independent?}}}	\\
		
		\vspace{0.8cm}
		Alexa Gopaulsingh\\
		
		\vspace{1mm}
		\singlespacing{	Department of Logic, ELTE}\\
		Budapest, Hungary\\
		
	\end{center}
	
	\vspace{1cm}

	\begin{center}
		\noindent \textbf{Abstract}
	\end{center}
	
	{\small In \cite{MainC},  Rosenmann and Ventura essentially asked "What is the right definition of dependence of subgroups for general groups?".  Here we aim to answer this question. We consider a definition of subgroup independence which is a special case of the category-theoretic one introduced in \cite{Z1}. It is that: Two subgroups of a group are \textbf{independent} if and only if any  two endomorphisms, one acting on each subgroup, can be extended to an endomorphism of the group generated by these subgroups.  This definition helps to illuminate that the usual condition of almost disjointness  of subgroups (two subgroups  $A$ and $B$  are almost disjoint if and only if $A \cap B = \{e\}$, where $e$ is the identity element) is not enough to force independence and here we find necessary and (different) sufficient conditions for subgroup independence. The aim of this note is to introduce this general notion of subgroup independence to the group theory community and to pose the open question of its characterisation. We present the partial results known up to this point.  Moreover, we use the progress made so far to give a heuristic algorithm that decides subgroup independence for many cases.}
	
	\vspace{5mm} 
	
	
	\\
	
	\begin{section}{Introduction}

		Independence notions have long been of interest in many fields. For example, they are important in physics in quantum field theory, see \cite{S1}. Here, many unequal notions of independence were introduced.  For a review and comparison of these, see \cite{S2}.  In \cite{MainC},  Rosenmann and  Ventura asked for a definition of independence which generalises the notion  first used by Rosenmann in \cite{MainC2} for free groups.  In \cite{Z1}, a general approach and a category-theoretic definition of subalgebra independence of structures was introduced. It generalises and continues work done in \cite{M1} and \cite{Z2} to give a categorical definition of independence as subobject morphism coexistence.  For the category of  Groups, this essentially reduces to: Two subgroups $A$ and $B$ of a group, are independent if and only if  any endomorphisms, $\alpha$ on $A$ and $\beta$ on $B$, can be extended to an endomorphism of the minimal group generated by $A$ and $B$ (their join). This extension, if it exists must be unique, see \cite{Z1}. Here, it uses the fact that in the categories of Sets, Boolean Algebras, Vector Spaces and Abelian Groups, the co-product of two algebras in these categories are the join of these algebras. This results in the independence definitions evaluating to the expected notions in these categories: For the category of  Sets, subset independence is disjointness; for Vector Spaces, subspace independence  is linear independence; for Boolean Algebras, subalgebra independence is logical independence and for the category of Abelian groups, subgroup independence is almost disjointness. However, surprisingly, for the category of Groups the situation was found to be mysterious and not equivalent to the usual requirement of almost disjoint subgroups. In this category, the internal co-product of two subgroups is not necessarily isomorphic to their join.  We found that it is possible for a pair of subgroups to satisfy the usual requirement of almost disjointness and yet still their endomorphisms affect each other. That is, there exists endomorphisms of almost disjoint subgroups of a group which do not have an extension in their join, see Example \ref{E1}. This category-theoretic definition helped to shed light on this situation and we see that we have to  be more careful and require more than simply almost disjointness to determine if the endomorphisms of two subgroups influence each other. Informally, we found that part of the reason for this is that even if  two subgroups $A$ and $B$ are such that $A \cap B = \{e\}$ (i.e they are almost disjoint), it still may be that case that $A$ contains a non-trivial conjugate of an element $B$ (or a product of $B$'s conjugates). Homomorphism mappings on elements restrict what their conjugates can be mapped to and hence $A$ can still affect $B$ in this way. This prompts us to define a notion of subgroup separateness which states that not only is $A$ almost disjoint from $B$ but it must be almost disjoint from the minimal normal group containing $B$ in their join. We found that requiring both $A$ to be separate from $B$  and  for $B$ to be separate from $A$, is indeed necessary for subgroup independence. However, we will show that even this is not sufficient. Additionally, we will see that if we were to take a stronger version of this separateness for $A$ and $B$ which states that their minimal normal subgroups in their join must be almost disjoint, then this is too strong to characterise subgroup independence. Though we have gotten closer, the exact sweet spot of this conjugate influence is yet to be found and we motivate it as an interesting open problem.
		
		Moreover, this article contains the progress so far, for determining when two subgroups are independent. The necessary and (separately) sufficient conditions in the following sections can help us to decide independence or dependence of subgroups for many cases. We have used these results to condense a heuristic algorithm given in Section \ref{S5} which can be useful to scientists or anyone who may need to detect subgroup dependence. \\

	\end{section}

	\begin{section}{Subgroup Independence}
		
		We start by giving the notations and conventions used in this article and we also introduce the subgroup separateness and subgroup independence definitions which we investigate.
		
		\begin{subsection}{\large Definitions and Notation}	\normalfont
			
			\noindent -- Let $A$ be a subset of a group $G$. Then, we denote the \textbf{subgroup generated by $A$} as \textbf{$\langle A\rangle $}.\\
			
			\noindent -- Let $A$ be a group. We will denote the \textbf{identity} of it by $e$.\\
			
			\noindent  -- $ X \cdot Y = \{x * y \  | x \in X \text{and} \ y \in Y\}$ for sets $X$ and $Y$ of a group $(G, *)$.\\
			
			\noindent -- The order of an element $g$ in a group $G$, will be denoted by $|g|$.\\
			
			\noindent -- Let a group $G$ be isomorphic to a group $H$. Then, we write $G \cong H$.\\

			\noindent -- Two subgroups, $A$ and $B$ are said to be \textbf{almost disjoint} iff $A \cap B = \{e\}$.\\
			
			\noindent -- Let $A$ and $B$ be subgroups of a group. Then, the group generated by their union, $\langle A \cup B\rangle $, is called the \textbf{join} of $A$ and $B$.\\
			
			\noindent --Let $A$ be a subgroup of a group $G$. We define $Conj(A)_{G}$ to be the\textbf{ set of conjugates} of $A$ in $G$.\\
			
			\noindent -- Let $A$ be a  group. Let \textbf{$\langle Conj(A)_G\rangle $} be \textbf{the minimal normal subgroup containing $A$ in $G$}. We  denote  this by $N_{A_G}$.\\

			\noindent -- In the following, we will always assume that $A$ and $B$ are subgroups of a common group and so it makes sense to consider their join. Also, as we will always be considering the $N_{A_{ \langle A \cup B\rangle }}$, we will simply write {\textbf{$N_A$}}.\\

			\noindent -- Let $(A, B)$ be a pair of groups. The endomorphisms $\alpha$ and $\beta$ of $A$ and $B$ respectively can be denoted by $(\alpha, \beta)$. If  $\alpha$ and $\beta$ can (cannot) be extended to an endomorphism of the $\langle A \cup B\rangle $, then we can say that $(A, B)$ is (is not) \textbf{compatible} with $(\alpha, \beta)$.\\
			
			\noindent --The identity and trivial endomorphisms will be of special interest to us and so we will use the abbreviations of \textbf{id} and\textbf{ triv} respectively to refer to these morphisms, where for any group $G$ and $g \in G$,  $\text{id}: G \rightarrow G$ is defined by $\text{id}(g) = g$ and $\text{triv}: G \rightarrow G$ is defined by $\text{triv}(g) = e$. For example, if we want to say that  the identity endomorphism  on $A$ is compatible with the trivial endomorphism on $B$, we will write that \textbf{$(id, triv)$ is compatible with $(A, B)$}.\\

			\noindent -- Cycle multiplication: In this article, there will be some examples using permutation groups. Here, we use the convention of considering the cycles as composition of functions and so we apply the multiplication of cycles from right to left.\\

			{\large\noindent \textbf{Separatedness Definitions:}}\\
			
			\noindent -- Suppose that $a  \in A$. 
			Then,  \textbf{an element $a$} is said to be \textbf{$B$-separated} iff  $a  \not\in  N_B $ or $a = e$.\\
			
			\noindent -- \textbf{A set $A$,} is called \textbf{$B$-separated }iff  $ \forall  \ a  \in A, a$ is  $B$-separated iff $ \forall  a \not= e \in A, \ a \not\in  N_B $. That is, iff $A\  \cap  N_B  = \{e\}$. \\

			\noindent -- \textbf{A pair (a, b)} $\in A \times B$  is said to be \textbf{separated} iff  and $a \not\in N_{\langle b\rangle} $ or $a = e$ and $b \not\in N_{\langle a \rangle} $ or $b = e$.\\
			
			\noindent	Next, we state the subgroup independence definition which is a special case of the subalgebra independence definition given in \cite{Z1}.\\

			{\large\noindent\textbf{Subgroup Independence Definition:} }\\
			
			\noindent Let \textbf{$(A, B)$ be a pair of subgroups}. If any pair endomorphisms ($\alpha$ ,$\beta$) of $(A, B)$ can be extended to a endomomorphism of their join $\langle A \cup B\rangle $, then we say that $A$ and $B$ are \textbf{independent}, which is denoted by $A \indep B$. Otherwise, we say that $A$ and $B$ are \textbf{dependent.}

		\end{subsection}
		
		\begin{subsection}{Main Results about Subgroup Independence}
			
			Next, we develop the main results obtained so far about subgroup independence. We will see  how the notion of subgroup separateness influences independence and we will obtain characterisations for subgroup separateness in this and the following sections.
			
			\begin{proposition}\label{P2.1}
				
				Let $A$ and $B$ be subgroups of a group. For any endomorphisms $\alpha$ and $\beta$ of $A$ and $B$ respectively, if their extension to an endomorphism exists in their join,   $\gamma$ say, then for any $\prod_{i= 1}^{n}a_ib_i \in \langle A \cup B\rangle $, it is the case that $ \gamma(\prod_{i=1}^{n}a_ib_i) =\prod_{i=1}^{n} \alpha(a_i)\beta(b_i)$.\\
			\end{proposition}
			
			\begin{proof}
				Immediate from the notion of extension of mappings.
			\end{proof}

			\begin{lemma}\label{L2.1}
				$(id, triv)$ is compatible for $(A, B) \implies$ $A$
				is $B$-separated.
			\end{lemma}
			
			\begin{proof}
				We will use proof by the contrapositive. Assume that $A$ is not $B$-separated. Then $\exists \  a \in A$ such that $a$ is not $B$- separated. This means that $a \not = e$  and  $a \in  N_B $. Consider $(\alpha, \beta) =(id, triv)$ for $(A, B)$.  We observe that, all elements in $Conj(B)$ must be mapped to $e$ in a homomorphism extension of $\beta$ since in any homomorphism, conjugate elements in the domain must be mapped to conjugate elements in the range and $e$ is conjugate only to itself. It follows that every element of $N_B $ must be mapped to $e$ since these elements can be expressed as a product of elements in $Conj(B)$. Since $a \in  N_B $, taking the triv homomorphism on $B$ implies that in the extension homomorphism, $a$ must be mapped to the identity in the join. However, by the id homomorphism, $a$ should be mapped to itself in the join. This implies that $a  = e$ which is a contradiction. Hence, $(id, triv)$ is not compatible for $(A, B)$. 
			\end{proof}

			\noindent Next, we use this lemma to show that a stronger condition than almost disjointness is required for independence.\\

			\begin{theorem}\label{mainTheorem}
				$A \indep B \implies$   $A\  \cap N_B  = \{e\}$ and  $B\  \cap N_A  = \{e\}$.\\
				
				\noindent Equivalently, $(A, B)$ is an independent pair $ \implies$ $A$ is $B$-separated and $B$ is $A$-separated.
			\end{theorem}
			
			\begin{proof}
				Applying Lemma \ref{L2.1}, we get that if $A$ is not $B$-separated that $(id, triv)$ is not compatible with $(A, B) \implies A \not \indep B$. Similarly, if $B$ is not $A$-separated then (triv, id) is not compatible with $(A, B)$ and so the pair is dependent.
			\end{proof}
			
			\noindent This gives us the (expected) corollary that almost disjointness is necessary for independence.
			
			\begin{corollary} \label{C2.1}
				$A \indep B \implies A \cap B = \{e\}$
			\end{corollary}
			
			\begin{proof}
				Immediate from the theorem. 
			\end{proof}	
			
			\begin{corollary}\label{useful-form}
				If there is a non-separated pair $(a, b)$ in $A \times B \implies A \not \indep B$.
			\end{corollary}
			
			\begin{proof}
				Immediate from the theorem. 
			\end{proof}

			\noindent The above necessary condition for independence of groups gives us a very useful way to tell if two groups are not independent. Consider the following example which shows that two groups can be dependent despite the fact that they are almost disjoint:
			
			\begin{example} \label{E1}
				\normalfont
				We consider subgroups of $S_3$, the symmetric group on 3 elements defined by; $A = \{e, (12)  \}$ and $B = \{ e,  (13)   \}.$ Now, $\langle A \cup B\rangle  =  S_3$. We see that $(12) \in A$ is conjugate to $ (13) $ in $S_3$. That is, $(12) \in \ \langle N_B \rangle  (= S_3)$ (hence $(id, triv)$ is not compatible with $(A, B)$). Therefore, $((12), (13))$ is not a separated pair and so $A$ and $B$ are not independent groups (notice that this is so, even though $A \cap B = \{e\}$).
			\end{example}
			
			\noindent We now state a  proposition which we will apply that follows from $A$ and $B$  not being independent.
			
			\begin{proposition}\label{notIndep}
				Let $A$ and $B$ be subgroups such that $A \not\indep B$, then $\exists$ homomorphisms $\alpha: A \rightarrow A$ and $\beta: B \rightarrow B$ such that for some $ \prod_{i = 1}^{n} a_ib_i = e \in \langle A \cup B\rangle $, where $a_i \in A$ and $b_i \in B$ for $i \in \{1, \dots n \}$ and it is the case that $ \prod_{i = 1}^{n} \alpha(a_i)\beta(b_i) \not = e.$
			\end{proposition}
			
			\begin{proof}
				Let $A$ and $B$ be subgroups such that $A \not\indep B$, then $\exists$ homomorphisms $\alpha: A \rightarrow A$ and $\beta: B \rightarrow B$ such that for some $\prod_{i = 1}^{l} a_ib_i , \  \prod_{i = 1}^{m} c_id_i  \in \langle A \cup B\rangle $, where $a_i \in A$ and $b_i \in B$ for $i \in \{1, \dots l \}$ and $c_i \in A$ and $d_i \in B$ for $i \in \{1, \dots m \},$ is such that, \\
				
				$\ \ \prod_{i = 1}^{l} a_ib_i  = \prod_{i = 1}^{m} c_id_i $ but  $\prod_{i = 1}^{l} \alpha(a_i)\beta(b_i) \not= \prod_{i = 1}^{m} \alpha(c_i)\beta(d_i) $ by Proposition \ref{P2.1}.  \\ $ \implies \prod_{i = 1}^{l} a_ib_i \prod_{i = 1}^{m} d_i^{-1}c_i^{-1} = e$ and  $ \prod_{i = 1}^{l} \alpha(a_i)\beta(b_i)\prod_{i = 1}^{m} \beta(d_i^{-1})\alpha(c_i^{-1})  \not = e$.\\
				
				\noindent Let $b_{l+1} = b_nd_m^{-1}$, $b_{l+j} = d_{m-(j-1)}^{-1}$ for $j \in \{2, \cdots m-1\}$ and $b_{l+m} = e$. \\
				Let  $a_{l+j} =  c_{m-(j-1)}^{-1}$ for $j \in \{1, \cdots m\}$.\\
				Now take $n= l+m$.\\
				
				\noindent Then, $\prod_{i=1}^{n}a_ib_i= e$ but $\prod_{i = 1}^{n} \alpha(a_i)\beta(b_i) \not= e$.
			\end{proof}
			
			\noindent Next, we obtain a useful equivalence to the condition that $A$ is $B$-separated.
			
			\begin{theorem} \label{T2.2}  
				
				$A$ is $B$-separated $\iff$ $(id, triv)$  is compatible with $(A, B)$.

			\end{theorem}	
			
			\begin{proof} ``$\Leftarrow$" direction is shown by Lemma \ref{L2.1}.
				
				``$\Rightarrow$" direction. We will show this by the contrapositive. Consider $(id, triv)$ for the pair of subgroups $(A, B)$. Suppose that $(id, triv)$ does not have a homomorphism extension in $\langle A \cup B\rangle .$ Then, by Proposition \ref{notIndep}, $\exists \ \prod_{i = 1}^{n} a_ib_i 
				= e \in \  \langle A \cup B\rangle $ which is such that $\prod_{i = 1}^{n} \text{id}(a_i)\text{triv}(b_i) \not = e$. Using the definitions of these mappings we get that $\prod_{i = 1}^{n}\text{id}(a_i)  = \text{id}(a_1 \dots a_n) \not= e$. Since $\alpha$ is a homomorphism this implies that $a_1 \dots a_n \not= e \implies a_n^{-1} \dots a_1^{-1} \not= e $ and it is clearly in $A$. Call it $a_k$, say.  \\
				
				\noindent Consider now, $ (a_1b_1 \dots a_nb_n)(a_n^{-1} \dots a_1^{-1}) = e \cdot a_k = a_k$.\\ This can be viewed as;\\
				
				$ \ (a_1(b_1(a_2 \dots (a_{n-1}(b_{n-1}(a_nb_na_n^{-1})) a_{n-1}^{-1}) \dots a_2^{-1})a_1^{-1})$.\\
				
				Each term in the brackets is either a conjugate of an element in $N_B $ or a product of elements of $B$ by an element in $N_B $ and so is also in $N_B $. Therefore, $a_k \in \  N_B $ but $a_k \not = e$. This implies that $A \  \cap  N_B  \not = \{e \}$ and $A$ is not $B$-separated. 	
			\end{proof}

			\begin{lemma}\label{productsInJoin}
				If $A$ and $B$ are subgroups such that  $A$ is $B$-separated and $B$ is $A$-separated then, 
				
				$\forall \  \prod_{i=1}^{n}a_ib_i = e \in \ \langle A \cup B\rangle $, it is the case that $\prod_{i=1}^{n}a_i = e$ and $\prod_{i=1}^{n}b_i = e$.
			\end{lemma}
			
			\begin{proof}
				Suppose that $A$ is $B$-separated. Then for any $ \prod_{i=1}^{n}a_ib_i = e \in \ l\langle A \cup B\rangle $ we get that,  $ \alpha(e) = \alpha(\prod_{i=1}^{n}a_ib_i) =  \prod_{i=1}^{n}\alpha(a_i)\beta(b_i) = e.$ By assumption and using  Theorem \ref{T2.2}, we know that $(id, triv)$ is compatible with $(A, B)$ and so let $\alpha$ = id and $\beta$ = triv. This gives us that $\prod_{i=1}^{n}\text{id}(a_i) = e \implies \text{id}(a_1 \dots a_n) = e \implies a_1 \dots a_n = e$.
				
				Similarly, $B$ is $A$-separated implies that  $\forall \ \prod_{i=1}^{n} a_ib_i= e \in \ \langle A \cup B\rangle $, it is the case that  $\prod_{i=1}^{n}b_i= e$.
			\end{proof}

			\begin{lemma}\label{normalLemma1}
				If $A$ and $B$ are subgroups such that $A \cap B = \{e\}$ and $A$ and $B$ are normal in their join, then $A$ is $B$-separated and $B$ is $A$-separated.
				
			\end{lemma}

			\begin{proof}
				Suppose the conditions. Since $A$ and $B$ are normal in the join, then $A = N_A $ and $B = N_B $. As $A \cap B = \{e\}$, this gives us that $N_A  \cap  N_B  = \{e\}.$ This implies that $A\  \cap N_B  = \{e\}$ and  $B\  \cap N_A  = \{e\}$, hence $A$ is $B$-separated and $B$ is $A$-separated.
			\end{proof}
			
			\begin{corollary}\label{normalLemma2}
				If $A$ and $B$ are such that $A \cap B = \{e\}$ and $A$ and $B$ are normal in their join then $\forall \  \prod_{i=1}^{n}a_ib_i = e \in \ \langle A \cup B\rangle $, it is the case that $\prod_{i=1}^{n}a_i = e$ and $\prod_{i=1}^{n}b_i= e$.
				
			\end{corollary}
			
			\begin{proof}
				Using Lemma \ref{productsInJoin} and Lemma \ref{normalLemma1}, we get the result.	
			\end{proof}

			\noindent We need the following proposition which is known in group theory. It states that the elements commute between almost disjoint normal subgroups.
			
			\begin{proposition}\label{normal-subgroups-and-commuting-elements-1}
				If $A$ and $B$ are normal subgroups of a group such that $A \cap B = \{e\}$, then  for any $a \in A, b\in B , ab = ba$.
			\end{proposition}
			
			\begin{proof}
				Let $a \in A$ and $b \in B$. Consider $aba^{-1}b^{-1}$. This is in $A$ since $aba^{-1}b^{-1} = a(ba^{-1}b^{-1})$ and $A$ is normal. Similarly, $aba^{-1}b^{-1} \in B$ since $aba^{-1}b^{-1} = (aba^{-1})b^{-1}$ and $B$ is normal. However, $A \cap B = \{e\}.$ This gives us that $aba^{-1}b^{-1} = e$ and so $ab = ba$.
			\end{proof}
			
			\noindent If we consider $A$ and $B$ as subgroups in their join, then we get a type of converse to the above. This states that if the elements commute between $A$ and $B$, then $A$ and $B$ are each normal in their join.
			
			\begin{proposition}\label{normal-subgroups-and-commuting-elements-2}
				If for any $a \in A$, $b \in B$, it is the case that $ab = ba$, then $A$ and $B$ are normal subgroups of their join. 
			\end{proposition}	
			
			\begin{proof}
				Let $a_k \in A$ and $a_1b_1 \dots a_nb_n \in \ \langle A \cup B\rangle $. Then, \\ $(a_1b_1 \dots a_nb_n)a_k(a_1b_1 \dots a_nb_n)^{-1} = a_1b_1 \dots a_nb_n a_k b_n^{-1} a_n^{-1} \dots b_1^{-1}a_1^{-1}$ \\ $ \textcolor{white}{.} \ \ \ \ \ \ \ \ \ \ \ \ \ \ \ \ \ \ \ \ \ \ \ \ \ \ \ \ \ \ \ \ \ \ \ \ \ \ \ \ \ = a_1 \dots a_n a_k b_1 \dots b_n b_n^{-1} \dots b_1^{-1} a_n^{-1} \dots a_1^{-1}$  since by \\  assumption, the $a_i$s and $b_i$s commute for $ i \in \{1 \dots n, k\}$. After reducing, this becomes \\ $a_1 \dots a_n a_k a_n^{-1} \dots a_1^{-1}$ which is in $A$ and hence $A$ is a normal subgroup of $\langle A \cup B\rangle $. The proof is similar for $B$.
			\end{proof}
			
			\noindent Overall, we get the following characterization which states that two almost disjoint subgroups are normal in their join if and only if the elements commute between them.
			
			\begin{lemma}\label{normalsCommute}
				If $A$ and $B$ are subgroups such that $A \cap B = \{e\}$, then they are normal in their join, if and only if, for any $a \in A$ and $b \in B$, $ab =ba$.
			\end{lemma}
			
			\begin{proof}
				
				This follows by combining Proposition \ref{normal-subgroups-and-commuting-elements-1} and Proposition \ref{normal-subgroups-and-commuting-elements-2}.
			\end{proof}
			
			\noindent The next result gives an important sufficient condition for independence.
			
			\begin{theorem}\label{normalTheorem}
				If $A$ and $B$ are subgroups which are normal in their join such that $A \cap B = \{e\}$, then $ A \indep B$.
			\end{theorem}
			
			\begin{proof} We proceed using proof by contradiction. Suppose that $A$ and $B$ are normal subgroups in their join but  $A \not\indep B$. Then by Proposition \ref{notIndep}, there exists homomorphisms $\alpha : A \rightarrow A$ and $\beta : B \rightarrow B$ such that for some $\prod_{i=1}^{n}a_ib_i = e \in \ \angle A\cup B\rangle  $ and it is the case that $\prod_{i=1}^{n}\alpha(a_i)\beta(b_i) \not = e$. By Lemma \ref{normalLemma1} and Lemma \ref{productsInJoin}, we know that $\prod_{i=1}^{n}a_i = e$ and $\prod_{i=1}^{n}b_i = e$.  Now by Lemma \ref{normalsCommute} the elements of $A$ commute with the elements of $B$, so by rearranging terms we get that $\prod_{i=1}^{n}\alpha(a_i)\beta(b_i) = \prod_{i=1}^{n}\alpha(a_i)\prod_{i=1}^{n} \beta(b_i)$. Now, since $\alpha$ and $\beta$ are homomorphisms this implies that $ \prod_{i=1}^{n}\alpha(a_i)\prod_{i=1}^{n} \beta(b_i)= \alpha (\prod_{i=1}^{n}a_i) \beta(\prod_{i=1}^{n}b_i) = e $ as  $\prod_{i=1}^{n}a_i = e$ and $\prod_{i=1}^{n}b_i = e$. This gives us that,  $\prod_{i=1}^{n}\alpha(a_i)\beta(b_i) = e$, which is a contradiction.
			\end{proof}	
			
			\begin{corollary}
				If $A$ and $B$ are subgroups of an abelian group such that $A \cap B = \{e\}$, then $A \indep B$.
			\end{corollary}
			
			\begin{proof}
				If $A$ and $B$ are subgroups of an abelian group, then this implies that their join is abelian and so they are normal in the join and we can apply  Theorem \ref{normalTheorem}.
			\end{proof}
			
			\noindent Note that the above corollary does not mean that any two abelian groups are independent, see Example \ref{E1} for two non-independent abelian groups (however, they are not subgroups of an abelian group). This result shows that for the category of Abelian Groups, almost disjointness is enough to force independence.\\

			\noindent \textbf{Commuting Pair check to Decide Independence:}\\
			
			The previous results give us the following very useful statement which can help us decide independence in some cases. It states that if $A$ and $B$ are almost disjoint subgroups whose elements commute between each other, then they are independent.

			\begin{theorem}\label{normalsC}
				
				If $A$ and $B$ are subgroups such that $A \cap B = \{e\}$ and for any $a \in A$ and $b \in B, \ ab =ba \implies A \indep B$. 
			\end{theorem}
			
			\begin{proof}
				
				Combining Lemma \ref{normalsCommute} and Theorem \ref{normalTheorem} gives the result.
			\end{proof}

			This leads us to the next result, which shows us that not only is it the case that almost disjoint normal subgroups, $A$ and $B$ say, are independent, but that for any subgroups $A' \leq A$ and $B' \leq B$  will also be independent. Not only that, but they will be normal in their join regardless of whether the subgroups were chosen to be normal in $A$ and $B$ respectively.
			
			\begin{proposition}
				Let $A$ and $B$ be subgroups such that $A \cap B = \{e\}$ and they are normal in $\langle A \cup B\rangle $. Then for any subgroups $A' \leq A$ and $B' \leq B$, $A'$ and $B'$ are normal in $\langle A' \cup B'\rangle $ and so are independent. 
			\end{proposition}
			
			\begin{proof}
				Suppose that $A$ and $B$ be subgroups such that $A \cap B = \{e\}$ and they are normal in $\langle A \cup B\rangle $. Let $A' \leq A$ and $B' \leq B$. Then by Lemma \ref{normalsCommute}, we get that for any $ a\in A$ and $b \in B$ that $ab=ba$. This implies that for any $a \in A'$ and $b \in B'$ that $ab = ba$ as $A' \subseteq A$ and $B' \subseteq B$ . Since $A \cap B = \{e\}$, we also get that $A' \cap B' = \{e\}.$ Then by Lemma \ref{normalsCommute}, it follows that $A'$ and $B'$ are normal in $\langle A' \cup B'\rangle $ and so by Theorem \ref{normalTheorem}, they are independent.
			\end{proof}

			\begin{theorem} \label{ppp}
				If $A$ and $B$ are subgroups such that $A$ is a normal in their join but 
				$B$ is not, then $A$ and $B$ are not independent.
			\end{theorem}	
			
			\begin{proof}
				
				Assume to get a contradiction that in their join $A$ is normal and $B$ is not but that $A \indep B$. Since $B$ is not normal in $\langle A \cup B\rangle $, then by Lemma \ref{normalsCommute}, there exists $a \in A, \ b \in B$ such that $ab \not= ba \implies b^{-1}ab \not= a$. Since $A$ is normal in $\langle A \cup B\rangle $, then $b^{-1}ab = a_k$ for some $a_k \not= a$. Consider the morphisms, $(id, triv)$ of the pair $(A, B)$.  As $A \indep B$, then they can be extended in their join to $\gamma$ say, hence we get that $\gamma(bab^{-1}) = \gamma(a_k) \implies \text{triv}(b)\text{id}(a)\text{triv}(b^{-1}) = \text{id}(a_k) \implies a = a_k$, which is a contradiction. Hence, their join morphism does not exist and $A$ and $B$ are not independent. 
			\end{proof}

			\noindent We therefore notice that, two subgroups which are normal in their join are always independent and if there is a pair $(A, B)$ such that one is normal in their join and the other is not, then this pair is never independent. This might prompt us to think that two subgroups which are both not normal in their join  would also never be independent. However, interestingly, this is not the case as the following examples show:

			\begin{example}\label{mainEg}
				\normalfont
				Let $A =  \{ e, (1 2)  \}$ and $B  = \{ e, (1 3) (2 4)\}$ which are subgroups of the symmetric group, $S_4$.\\ 
				Then $\langle A \cup B\rangle  = \{ e, (1 2), (3 4) , (1 2)(3 4), (1 3)(2 4), (1 4)(2 3), (1324), (1423)\}$.  Now,\\
				$N_A  = \{e, (12), (34), (12)(34)\}$ and\\
				$N_B  = \{ e, (1 3)(24), (12)(34), (14)(23)\}.$\\ Notice, that $A$ is $B$-separated and $B$ is $A$-separated, that is;\\
				$A \  \cap  N_B  = \{e\}$ and \\
				$B \  \cap N_A  = \{e\}.$
				
				\noindent This implies that $(id, triv)$ and (triv, id) are compatible for $(A, B)$ by Theorem \ref {T2.2}. Since the identity and trivial endomorphisms are the only endomorphisms of $A$ and $B$  then this implies that  $A \indep B$. Also, since $A \not= N_A $ and $B \not =  N_B $, then it can  be seen $A$ and $B$ are independent even though they are each not normal in their join.

			\end{example}
			
			\begin{example}\label{ZalansFavourite}
				\normalfont
				Consider the group $D_{\infty}$ given by the presentation $D_{\infty} = \langle  x, y \mid x^2=y^2=e\rangle $. Let $A = \langle x\rangle $ and $B = \langle y\rangle $ be the 
				subgroups generated  by $x$ and $y$ respectively. Now, \\
				
				\noindent $N_A  = \{z \in D_{\infty} \ | \ z \text{ has an  even number of $y$'s} \}$ and \\
				$N_B  = \{z \in D_{\infty} \ | \ z \text{ has  an even number of $x$'s} \}$.\\
				
				\noindent Observe that $A \  \cap N_B  = \{e\}$ and  $B \  \cap N_A  = \{e\}$. This implies that $(id, triv)$ and (triv, id) are compatible for $(A, B)$ by Theorem \ref {T2.2}.  As $A$ and $B$ only have these endomorphisms, then this implies that  $A \indep B$. Also, since $A \not= N_A $ and $B \not = N_B $, then it can  be seen $A$ and $B$ are independent even though they are each not normal in their join.
				
			\end{example}

			\noindent \textbf{Non-commuting Pair Check to Decide Dependence:}\\
			
			Now, we know by Theorem \ref{normalsC} and that if two subgroups $A$ and $B$ are such that for all $a \in A$ and $b \in B$, $ab=ba$, then $A$ and $B$  are independent subgroups. So, for a pair of subgroups to be dependent there must exist an $ a \in A$ and $b \in B$ such that $ab \not= ba$. However, this is not sufficient for non-independence as Example \ref{mainEg} illustrates. There, $A = \{e, (12)\}$ and $B = \{e, (13)(24)\}$ were shown to be independent, however (12) and (13)(24) do not commute. Nonetheless, as the next result shows, in many cases, a quick of check pairs of non-commuting elements can help us detect non-independence.

			\begin{proposition}\label{OrderTheorem}
				Let $A$ and $B$ be subgroups such that $\exists \ a \in A$ and $b \in B$ and $ab \not=ba$. If $|ab|$ is finite and   $|a| \not | \ |ab|$ or $|b| \not | \ |ab|$, then $A \not \indep B$.
			\end{proposition}
			
			\begin{proof}
				Suppose, without loss of generality that $|a| \not | \ |ab|$.\\

				Consider the homomorphism pair $(id, triv)$for $(A, B)$.\\
				
				Now, $(ab)^{|ab|} = e$.\\

				If the extension of $(id, triv)$ exists for $(A, B)$, we get that,\\
				
				$(\text{id}(a)\text{triv}(b))^{|ab|} = \text{id}(e).$\\
				
				$\implies (\text{id}(a))^{|ab|}  = e$.\\
				
				$\implies a^{|ab|} = e$.\\
				
				However, this contradicts that $|a| \not | \ |ab|$. Therefore, the extension of $(id, triv)$ does not exist for $(A, B)$ and  $A$ is not independent to $B$.    
			\end{proof}
			
			\noindent We now show how this theorem can be used to quickly decide non-independence of subgroups. 
			
			\begin{example}\label{E2}
				\normalfont
				Let $A$ and $B$ be subgroups of $S_3$, the symmetric group on 3 elements, defined by $A = \{e, (12)\}$ and $B = \{e, (123), (132)\}$. Consider $a = (12)$ and $b = (123)$ (notice that these two elements do not commute). Then $ab = (12) (123) = (23)$. Therefore, the order of $b$ is 3 and the order of $ab$ is 2 and so $ |b| \not| \ |ab|$. Hence, $A$ and $B$ cannot be independent groups.
			\end{example}
			
			\noindent \textbf{Remark 1:} Notice that for commuting elements $a \in A$, $b\in B$ of finite orders, it is always the case that $|a|\  | \ |ab|$ and $|b| \ | \ |ab|$. So, this check for non-independence is (as expected) to be performed on non-commuting elements. In the case of Example \ref{mainEg}, even though $A = \{e, (12)\}$ and $B = \{e, (13)(24)\}$ are independent with non-commuting elements $a= (12) \in A$ and $ b=(13)(24) \in B$, we observe that $|(12)| = |(13)(24)| = 2$  and $|(12)((13)(24))|=|(1324)|=4$. So in this case, $|a| \ | \ |ab|$ and $|b| \ | \ |ab|$. However, it can be seen that this check would work to decide non-independence in "most" such cases with a non-commuting pair. This "most" is loose because we are comparing infinite sets. However, we may conjecture that if we were to linearly order the finite permutation groups by lexicographical ordering of their elements, then the number of counterexample groups (groups with non-commuting pairs which are nonetheless independent) in the first $n$ groups would be low for all $n$. Even, if this conjecture is not true, this check would still work for  many (infinite) cases, so this will be an important step in our heuristic algorithm in Section \ref{S5}.\\
			
			\noindent \textbf{Relationship between Commuting and Independence:}\\
			
			\noindent Let us say that the pair of subgroups $(A, B)$ \textbf{commutes} if $AB=BA$. Overall, the relationship between commuting and independence is as follows:  \\
			
			\noindent  $A \cap B = \{e\}$ and $(A,B)$  commutes  $\implies (A,B)$ is independent.\\
			
			\noindent \textbf{Remark 2}: The converse is "almost" but not actually true for the reasons just described. Another point is that, informally, a reason that the existence of a non-commuting pair does not automatically imply non-independence is that the groups may not be big enough for the homomorphisms on them to detect the "dependence" that non-commuting (non-free)  pairs intuitively indicate. In this case, enlarging the groups can detect the independence. See, the enlargement of Example \ref{mainEg} to  Example \ref{mainEg2} (mentioned later as an important counterexample). The "dependence" of the non-commuting pair of elements $((12), (13)(24))$ was not detected in the first example but in the enlarged second example, there are enough homomorphisms for this pair  to be a problem (i.e. there exists a pair of endomorphisms $(\alpha, \beta)$ which is not compatible with $(A, B)$).\\
			
			\noindent \textbf{Change of Conjugate Check For Detecting Dependence:}
			
			\begin{lemma}\label{ConjLemma}
				
				\noindent  $A$ is $B$-separated  then 
				there are no $\ a_1, a_2 \in A$ such that $a_1$ and $a_2$ are not conjugates in $A$ but are conjugates in $\langle  A \cup B\rangle $. \\
			\end{lemma}
			
			\begin{proof}
				Suppose that $A$ is $B$-separated. Assume to get a contradiction that $\exists \ a_1, a_2 \in A$ such that $a_1$ and $a_2$ are not conjugates in $A $ but are conjugates in $\langle  A \cup B\rangle $.\\
				
				Then $\exists \ z = a_1b_1 \dots a_nb_n  \in \  \langle  A \cup B\rangle $ such that, $ a_1 = z a_2 z^{-1}$.\\
				
				That is, $a_1 = (a_1b_1 \dots a_nb_n)a_2 (b_n^{-1}a_n^{-1} \dots b_1^{-1}a_1^{-1})$.\\
				
				Since $A$ is $B$-separated, then by Theorem \ref{T2.2},  $(id, triv)$ is compatible with $(A, B)$.  Therefore, \\
				
				\noindent $\text{id}(a_1) = (\text{id}(a_1)\text{triv}(b_1) \dots \text{id}(a_n)\text{triv}(b_n)) \text{id}(a_2) (\text{triv}(b_n^{-1})\text{id}(a_n^{-1}) \dots \text{triv}(b_1^{-1})\text{id}(a_1^{-1}))$.\\
				
				\noindent $\implies a_1 = (a_1\dots a_n) a_2(a_n^{-1} \dots a_1^{-1}) = (a_1\dots a_n) a_2(a_1\dots a_n) ^{-1} \in A.$\\
				
				This contradicts the assumption that $a_1$ and $a_2$ are not conjugates in $A$.
			\end{proof}

			This gives us the following result which states that if $A$ and $B$ are independent, then conjugacy classes in $A$ cannot be joined in  $\langle  A \cup B\rangle $. Similarly, for conjugacy classes in $B$.

			\begin{theorem}\label{ConjCheck}
				$A \indep B \implies$\\
				
				\noindent (i) 
				$ \not\exists \ a_1, a_2 \in A$ such that $a_1$ and $a_2$ are not conjugates in $A$ but are conjugates in $\langle  A \cup B\rangle $ and \\
				
				\noindent (ii)	
				$ \not\exists \ b_1, b_2 \in A$ such that $b_1$ and $b_2$ are not conjugates in $B$ but are conjugates in $\langle  A \cup B\rangle $.
				
			\end{theorem}
			
			\begin{proof}
				As $A \indep B \implies$ implies that $A$ is $B$-separated and that $B$ is $A$-separated (by Theorem \ref{mainTheorem}), by applying Lemma \ref{ConjLemma} we get the result.
			\end{proof}

			\noindent This can be  useful for deciding non-independence as witnessed by the following example.

			\begin{example}
				\normalfont
				Let $A$ and $B$ be subgroups of $S_4$, the symmetric group on 4 elements defined by $A = \{e, (12), (34), (12)(34)\}$ and $B = \{e, (1234), (1432), (13)(24)\}$. It is easy to check that, $(12)$ and $(34)$ are not conjugates in $A$. Now, $\langle  A \cup B\rangle  = S_4$ since it is well known that $(12)$ and $(12\dots n)$ generates $S_n$. However, $(12)$ and $(34)$ \textit{are} conjugates in $\langle  A \cup B\rangle = S_n$ since they have the same cycle structure. Therefore, $A$ cannot be independent to $B$.
			\end{example}

		\end{subsection}

	\end{section}
	
	\begin{section}{Further Connections following from Subgroup Separateness}
		
		\begin{subsection}{Factoring the Join of Two Subgroups by One of the Subgroups}
			
			If $A$ is independent to $B$ we might expect that factoring their join by one group would give back  the other (and not something smaller)  as their join is the minimal group containing $A$ and $B$. Since, we cannot technically factor by a non-normal subgroup, what we will find, using the second homomorphism theorem, is that if their join is factored by the minimal normal subgroup of one, then that quotient group will be isomorphic to the other.  We will see that each subgroup being separated from each other is enough to get this result.\\
			
			\noindent We need following  known result in group theory, see \cite{B2}.
			
			\begin{proposition}\label{P2.5}
				Let $H$ and $N$ be subgroups of a group where $N$ is normal in $G$. Then $\langle  N \cup H\rangle  = N \cdot H$ and so $N \cdot H= H\cdot N$.

			\end{proposition}	
			
			\begin{corollary}\label{C2.5}
				
				Let $A$ and $B$ be subgroups. Then, $\langle  A \cup B\rangle  \ = N_A\rangle  \cdot B =  B \cdot N_A \ = \  N_B \cdot A = A \cdot  N_B $.
				
			\end{corollary}
			
			\begin{corollary}\label{C2.5}
				
				Let $A$ and $B$ be subgroups. Then, $\langle  A \cup B\rangle  \ = N_A  \cdot B =  B \cdot N_A \ = \ N_B  \cdot A = A \cdot N_B\rangle $.
				
			\end{corollary}
			
			\begin{proof}
				Immediate from Proposition \ref{P2.5}
			\end{proof}

			\begin{lemma}\label{FactoringL}
				Let $A$ and $B$ be subgroups. Then, \\
				
				$A$ is $B$-separated $\implies$ 	\large{$\frac{\langle  A\cup B\rangle }{N_B } \cong A$} 
				
			\end{lemma}
			
			\begin{proof}
				Applying, $A$ and $N_B $  in the second homomorphism theorem with $N_B $ being the normal subgroup, we get that, \\
				
				{\large	$\frac{A \cdot N_B }{N_B} \cong \frac{A}{A \cap N_B }$.}\\
				
				Since $A$ being $B$-separated means that $A \  \cap N_B  = \{e\}$ and using that $\langle  A \cup B\rangle  = N_B  \cdot A$ from Corollary \ref{C2.5}, we get that,
				{\large $\frac{\langle  A\cup B\rangle }{N_B } \cong \frac{A}{\{e\}} \cong A$}. Similarly, we can  show that, if $B$ is $A$-separated then,
				$\frac{\langle  A\cup B\rangle }{N_A} \cong B.$
			\end{proof}	
			
			\begin{theorem}\label{FactoringT}
				Let $A$ and $B$ be subgroups. Then, \\
				
				$A \indep B$ $\implies$ 	\large{$\frac{\langle  A\cup B\rangle }{N_B} \cong A$} \normalsize{and}  \large{$\frac{\langle  A\cup B\rangle }{N_A } \cong B$.}
				
			\end{theorem}
			
			\begin{proof}
				
				Applying Theorem \ref{mainTheorem} and Lemma \ref{FactoringL}, we get the result.
			\end{proof}
			
			\noindent If $A$ and $B$ are finite subgroups if a group, then we have the following converse to Lemma \ref{FactoringL}
			
			\begin{theorem}\label{T3.2}
				Suppose that  $A$ and $B$ are finite subgroups of a group. Then, \\
				
				{\large $\frac{\langle  A\cup B\rangle }{N_B } \cong A$} $\implies$  $A$ is $B$-separated.
				
			\end{theorem}
			
			\begin{proof}
				We use proof by contrapositive. Suppose that $A$ is not $B$-separated. Then, $A \ \cap  N_B  \not= \{e\}$ and $|A \ \cap  N_B | \rangle  1$ and so {\large$ |\frac{A}{A \cap  N_B}| \langle |A|$}. Therefore, {\large$\frac{A}{A \cap  N_B } \not \cong A$}. Since, {\large	$\frac{\langle A\cup B\rangle }{N_B } \cong \frac{A \cdot  N_B }{N_B } \cong \frac{A}{A \cap  N_B }$}, we get that {$\large \frac{\langle A\cup B\rangle }{N_B\rangle } \not \cong A$.}\\
			\end{proof}
			
			\noindent \textbf{Question}: \textit{Is this result still true if we consider all groups (infinite ones as well)?}
			
		\end{subsection}

		\begin{subsection}{Element-Writing Independence}
			
			In \cite{MainC2}, Rosenmann defined an element $g \in G$ as dependent on a subgroup $H$ of $G$ if $g$ can be written as a non-identity product of elements of $H$. Relatedly, note that $A$ being $B$-separated means that elements of $A$ cannot be written as a non-trivial combination of not only elelments in $B$ but elements in $N_B\rangle $ as $A \ \cap N_B = \emptyset$. While $A$ being $B$-separated and $B$ being $A$-separated  are necessary conditions for their independence by Theorem \ref{mainTheorem},  Example \ref{mainEg2} given in the  next section, shows that this is not sufficient to conclude that $A\indep B$. 
			
			Following this direction, \textbf{\textit{independent sets}} in a group will be defined as a set of elements such that no element can be written in terms of the other elements of the set:

			\begin{definition}
				Let $G$ be a group and $G' \subseteq G$. Then $G'$ is said to be an\textbf{ independent set} of $G$, if there does not exist $g_k \in G'$ which is such that $g_k= \prod g'_i$ where $g'_i \in G' - \{g_k\}$ for all $i$.
			\end{definition}
			
			\noindent Next, we give a result which shows that if $A$ and $B$ are independent subgroups, then the union of two independent subsets of these, one belonging to $A$ and the other to $B$, forms an independent set. 
			\begin{proposition}
				$A \indep B \implies$ \\
				if  $A' \subseteq A$, $B' \subseteq B$  are independent sets then  $A' \cup B'$ is an independent set.
			\end{proposition}

			\begin{proof}
				By contradiction. Assume that $A \indep B$ and  there exists $A' \subseteq A$, $B' \subseteq B$  which are independent sets but   $A' \cup B'$ is not an independent set.\\
				
				Then without loss of generality, there exists $a_k \in A' \cup B'$ and $a_k \in A'$ is such that  $a_k = \prod a_ib_i$ where $a_i \in A' - \{a_k\}$ and $b_i \in B'$ for all $i$. \\
				
				Consider the homomorphisms, $(id, triv)$ on $(A, B)$. Since $A \indep B$, a homomorphism extension of this exists and hence,
				
				$\text{id}(a_k) = \prod\text{id}(a_i)\text{triv}(b_i)$.\\
				$\implies\text{id}(a_k) = \prod \text{id}(a_i)$.\\
				$\implies a_k = \prod a_i$.
				
				Since $a_i \in A' - \{a_k\}$, this contradicts that $A'$ is an independent set.
			\end{proof}

		\end{subsection}
		
	\end{section}
	
	\begin{section}{The Open Problem of Subgroup Independence}

		Despite the nice connections discovered from the notion of a subgroup being separated from another, it was somewhat surprising that  it was found to not be strong enough to characterise subgroup independence. The following counterexample illustrates that while two groups being mutually separated from each other is necessary for them to be independent, it is not sufficient (this example is a modification of Example \ref{mainEg}):
		
		\begin{example}\label{mainEg2}
			\normalfont
			Let $A = \{e, (12)\} \cdot \{e, (56)\} =  \{ e, (1 2), (5 6), (1 2 )(5 6)  \}$ and \\ $B  = \{ e, (1 3) (2 4)\}$ which are subgroups of the symmetric group $S_6$.\\ Then $\langle A \cup B\rangle  = \{ e, (1 2), (3 4) , (1 2)(3 4), (1 3)(2 4), (1 4)(2 3), (1324), (1423)  \} \cdot \{e, (5 6)\}$. Now,\\
			$N_A  = \{e, (12), (34), (12)(34)\} \cdot \{e, (5 6)\}$ and\\
			$N_B  = \{ e, (1 3)(24), (12)(34), (14)(23)\}.$\\ Notice, that $A$ is $B$-separated and $B$ is $A$-separated, that is;\\
			$A \  \cap  N_B  = \{e\}$ and \\
			$B \  \cap N_A  = \{e\}.$\\
			However, consider the following homomorphisms, $\alpha: A \rightarrow A$ defined by $ \alpha(e) = e, \  \alpha((12)) = (56), \ \alpha((56)) = (12)$ and $\alpha((12)(56)) = (12)(56)$ and $\beta: B \rightarrow B$ where $\beta$ is the identity mapping.	If a join mapping exists, $\gamma$, then since $(56)((13)(24)) =  ((13)(24))(56)$, we should get that $\gamma((56)((13)(24))) = \gamma (((13)(24))(56)) $ which implies that $\alpha(56) \beta((13)(24)) = \beta((13)(24))\alpha(56)$.\\
			However,  	$\alpha(56) \beta((13)(24)) = (12)(13)(24)= (1324)$ whilst, $ \beta((13)(24))\alpha(56) = (13)(24)(12) = (1423)$. Therefore, $\gamma$ does not exist and these homomorphisms are incompatible. Hence, despite $A$ being $B$-separated and $B$ being $A$-separated, we observe that $A$ and $B$ are not independent.
		\end{example} 
		
		\noindent \textbf{Remark 3}: In the above example, we can informally see that while the elements $(12)$ and $(56)$ behave symmetrically from the point of view of $A$, $B$ however, can detect their difference. That is, $B$ "sees" that they are made from different numbers via the element $(13)(24)$. This element commutes with $(56)$ but not with $(12)$. Hence, the automorphism of $A$, $\alpha$, used above which swaps these two elements, cannot coexist with the identity mapping of $B$ in their join, as $B$ does not "view" these elements as indistinguishable. It is for this type of reason that if we consider independence as a relation between groups, then this relation is not invariant upon factoring groups by isomorphism. That is, let $A$ be isomorphic to $A'$ and $B$ be isomorphic to $B'$, then $(A, B)$ being an independent pair does not imply that $(A', B')$ is an independent pair. This can be witnessed by the following counterexample: Let $A = \{e, (12)\}, \ A' = \{e, (13)\}, \ B = \{e, (34) \}, \text{and} \ B'= \{ e, (34)\}$. Then, even though clearly $A \cong A'$, $B \cong B'$ and $(A, B)$ is an independent pair (as the elements between them commute, see Theorem \ref{normalsC}), however $(A', B')$ is not an independent pair (as $|13| \not| \ |(13)(34)| = |(134)|$, see Proposition \ref{OrderTheorem}). This is because isomorphisms lose information about the details of what particularly makes up groups and these details can influence independence. In other words, two groups can "locally" view themselves as the same when considered individually but a larger context (like in their join) can illuminate their differences and these can influence whether or not they affect each other in the wider context.
		
		So we see that that for two subgroups, $A$ and $B$ that satisfying the separateness conditions that $A \  \cap  N_B = \{e\}$ and $B \ \cap N_A  = \{e\}$, is not enough to force independence. At this point, one might be tempted to conjecture that the stronger separateness condition $N_A  \cap  N_B = \{e\} $ gives the characterisation but this still misses the sweet spot as the following considerations show:   Notice that for any $a \in A$ and $b \in B$, $aba^{-1}b^{-1} \in N_A  \cap  N_B $ since $aba^{-1}b^{-1}  = a(ba^{-1}b^{-1}) = (aba^{-1})b^{-1}$. So if  $N_A  \cap  N_B  = \{e\}$, then $ aba^{-1}b^{-1}  = e \implies ab = ba \ \forall \ a \in A \ \text{and} \ b \in B$. By Theorem \ref{normalTheorem} this means that $A$ and $B$ are each normal in their join and therefore independent. However, it is  not the case that  all pairs of independent subgroups are  normal in their join (see Example \ref{mainEg}, where $N_A  \cap  N_B  = \{e, (12)(34)\}$), and so this condition, whilst sufficient to conclude independence, is therefore not necessary for it. 
		
		Overall, we are left with the strange situation with respect to characterising that a subgroup pair  $(A, B)$, is independent:\\
		
		\noindent \textbf{Too loose} (Necessary but not sufficient): \\ $A \  \cap  N_B  = \{e\}$ and $B \  \cap N_A  = \{e\}$.\\
		
		\noindent \textbf{Too tight} (Sufficient but not necessary):\\ $N_A  \cap  N_B  = \{e\}$\\
		
		\noindent This prompts the question: \textit{\textbf{Is there something elegant in between those that would work?}}\\
		
		\noindent Hence, we pose the open problem to the community:\\
		
		\noindent \textbf{Open Problem:}\\ 
		
		\noindent \textit{What is a group-theoretic characterisation of subgroup independence?}\\

	\end{section}

	\begin{section}{What to Do if You Need to Determine if Two Subgroups Affect Each Other} \label{S5}
		
		\textbf{Heuristic algorithm for Detecting if Two Subgroups are Independent:}\\
		
		If given two groups $A$ and $B$, the following steps can be taken to test if they are independent:\\ 
		
		\textbf{1. }Check if there is an non- identity element in $A \cap B$. If yes, then $A$ and $B$ are not independent and you are done. If no, go to step 2.\\

		\textbf{2 (i).} Check if all pairs $a \in A$ and $b \in B$ are such that $ab = ba$. If yes, then $A$ and $B$ are independent and you are done. If no, then go to step 2 (ii).
		
		$\ $\textbf{ (ii).} Check if any of the  non-commuting pairs of $a \in A$ and $b \in B$, found in part (i), is such that $|a|  \not| \ |ab|$. If any such pair is found then, groups $A$ and $B$ are not independent and you are done.\\
		
		\textbf{3.} Calculate $\langle A \cup B\rangle $.  Then,  depending on the subgroups, do the easier of the following steps,  moving on to the next easier step if a decision is not found:\\
		
		\textbf{(i).} Calculate  $N_A $.  Check if there is a non-identity element in $ B \ \cap N_A $. If yes, then you are done and $A$ and $B$ are not independent. \\
		
		\textbf{(ii).} Calculate $N_B $. Check if there is a non-identity element in $ A \ \cap  N_B $. If yes, then you are done and $A$ and $B$ are not independent.\\
		
		\textbf{(iii).}  Check,  if $ \exists \ a_1, a_2 \in A$ such that $a_1$ and $a_2$ are not conjugates in $A$ but are conjugates in $\langle A \cup B\rangle $.\\
		
		\textbf{(iv).} Check if $ \exists \ b_1, b_2 \in B$ such that $b_1$ and $b_2$ are not conjugates in $B$  but are  conjugates $\langle A \cup B\rangle $.\\

		\textbf{4.} If none of the above steps yield a decision, then you are left with checking that the remaining of the endomorphisms of $A$ and of $B$ can be extended to form an endomorphism of their join $\langle A \cup B\rangle $ (if the above steps did not give a decision it means that $(triv, id)$ and $(id, triv)$  are compatible with your subgroups $(A, B)$. Hence, it is the remaining endomorphism pairs that need to be checked). This may have a high computational cost.\\
		
		\begin{proof}
			
			Step 1 is supported by Corollary \ref{C2.1}. Step 2 is supported by Theorem \ref{normalsC} and Proposition \ref{OrderTheorem}. Step 3 is supported by Theorem \ref{mainTheorem} and Theorem \ref{ConjCheck}.
		\end{proof}

		Steps 1 and 2 may be of relatively low computational cost to perform, depending on the groups. Steps 3 starts to become more costly, but again depending on the particular groups, there are many cases where this is still tractable. If steps 1-3 do not give a decision, then we are left to manually  check the remaining endomorphisms pairs at step 4. It is yet to be solved a complete characterisation, which would lower the computational complexity of step 4.  However, a great many cases (see Remark 1) can be decided in steps 1-3. so we hope that this algorithm can be nonetheless helpful. 
		
		Additionally, not only can the considerations in this paper can be used to determine if two groups are independent but it can be used to quickly design independent groups if so desired. As we have seen, simply requiring that $A \cap B =\{e\} $ is not enough. However, if in addition you make sure that the elements $a \in A$ and $b \in B$ commute for all $ a\in A$ and $b \in B$, then that would be enough. Also, if you wanted to design a pair of non-obviously dependent groups (i.e. with $A \cap B = \{e\}$), then giving $A$ and $B$ a non-commuting pair of elements $a \in A$ and $b \in B$ such that $|a| \not| \ |ab|$ would work.
		
		Overall, we hope to motivate further research in the direction of the complete characterisation of subgroup independence. 
		Also, for the scientists who work with groups and subgroups, we give the practical heuristic algorithm to determine when two subgroups affect each other in many cases.
		
	\end{section}
	
	\begin{section}{Acknowledgements}
		
		The author wishes to sincerely thank Zalán Gyenis for introducing the problem and his helpful comments and remarks on this manuscript.	
		
	\end{section}

	\bibliographystyle{plain}
	
	\bibliography{mybib11}
	
	\vspace{3 mm}
	
	\noindent \textsc{Department of Logic, Institute of Philosophy, Eötvös Loránd University, Budapest, Hungary}\\
	\textit{Email address:} \textbf{alexa279e@gmail.com}

\end{document}